\documentclass[12pt]{iopart}

\expandafter\let\csname equation*\endcsname\relax

\expandafter\let\csname endequation*\endcsname\relax
\usepackage{amsmath}

\usepackage{chngcntr}
\usepackage{apptools}
\AtAppendix{\counterwithin{lemma}{section}}
\AtAppendix{\counterwithin{thm}{section}}
\usepackage[titletoc]{appendix}

\usepackage{amssymb}
\usepackage{amsthm}
\usepackage{graphicx}
\usepackage{hyperref}
\usepackage[margin=0.75in]{geometry}
\newtheorem{thm}{Theorem}
\newtheorem{cor}{Corollary}
\newtheorem{lemma}{Lemma}

\usepackage{fancyhdr}
\usepackage{color}
\newcommand{\PP}{ \mathbf{Pr} }
\newcommand{\EE}{ \mathbb{E} }
\newcommand{\dt}{\delta t}
\newcommand{\ra}{\rightarrow}
\newcommand{\bv}[1]{\boldsymbol{\mathbf{#1}}}

\providecommand{\keywords}[1]{\qquad \quad \ \ \ \textbf{Keywords:} #1}

\begin{document}
\title[Memorised Random Walks]{The shape of a memorised random walk}

\author{Micha\l\  Gnacik, Abdulrahman Alsolami, and James Burridge}

\address{School of Mathematics and Physics, University of Portsmouth, Portsmouth, \\ PO1 3HF, United Kingdom}
\ead{\href{mailto:james.burridge@port.ac.uk}{\nolinkurl{james.burridge@port.ac.uk}},\href{mailto:michal.gnacik@port.ac.uk}{\nolinkurl{michal.gnacik@port.ac.uk}} }
\vspace{10pt}
\begin{indented}
\item[]August 2018
\end{indented}
\begin{indented}
	\item[] Published in \emph{J. Stat. Mech.} (2018) 083207 \\ \url{http://stacks.iop.org/1742-5468/2018/i=8/a=083207}
\end{indented}
\begin{abstract}
	We view random walks as the paths of foraging animals, perhaps searching for food or avoiding predators while forming a mental map of their surroundings. The formation of such maps requires them to memorise the locations they have visited. We model memory using a kernel, proportional to the number of locations recalled as a function of the time since they were first observed. We give exact analytic expressions relating the elongation of the memorised walk to the structure of the memory kernel, and confirm these by simulation. We find that more slowly decaying memories lead to less elongated mental maps. 
\end{abstract}

%
\vspace{2pc}
\keywords{Brownian motion, diffusion, exact results, stochastic processes}
%
%
%

\newpage
\tableofcontents

\section{Introduction}

Patterns of animal movement and the evolutionary forces which have given rise to them have long been of interest to biologists \cite{kar83,fag13,fag14}. The trajectories of many creatures are reminiscent of random walks \cite{tur15, cod08}, suggesting a link to statistical physics. In fact, collaboration between the two disciplines goes back many decades: insect movement patterns were first studied using correlated random walks in 1983 \cite{kar83}. Since then, the increasing availability of animal trajectory data \cite{col06} has motivated a great deal of modelling work by ecologists and physicists. More exotic classes of random walk have been considered as candidates for animal foraging trajectories. For example L\'{e}vy walks \cite{rap03,jam11,vis08,vis99}, have been demonstrated to be an optimal search strategy in certain environmental contexts \cite{hum10}. 
Modelling trajectories using Markov (memoryless) random processes \cite{law06} brings the advantage of mathematical tractability, but when movement patterns are guided by memory then more sophisticated processes emerge. For example some species, such as bees \cite{tho96} and hummingbirds \cite{gill88}, follow repeated foraging routes through their territory, known as traplines \cite{oha05}. Other species, such as humans and monkeys, make repeated long distance relocations to previously visited sites \cite{boy14,boy12}, and simple models of such behaviour exhibit slow diffusion and scale free occupation patterns. From an evolutionary perspective, one must consider the costs and benefits of spatial memory when considering its impact on animal movement \cite{fag13}. The ability to recall hiding places and food caches, or to accurately navigate long distances, must be balanced against the physiological and metabolic cost of a larger brain \cite{duk99}, and the need for sleep \cite{rot10}. Field studies suggest that higher primates such as chimpanzees possess Euclidean mental maps of their environment, and are able to select the most direct route when returning to a distant food source, rather than relying on repeated paths or landmarks \cite{nor09}. The maintenance of such maps requires a spatial memory which must be reinforced and maintained. The decay of this memory may strongly effect movement patterns through a landscape \cite{avg13}. This relationship between memory, its rate of decay, and the shape of a mental map is the focus of our paper. Using methods originally devised to characterise the shapes of long chain polymer molecules \cite{rud86}, we relate the \textit{asphericity} (elongation) of the geographical region recalled by a random walking forager, to the process by which its memory decays \cite{ave11}. If one views the trajectory of the forager as a search path, then such measures provide a description of the effective search area of the animal.  
We note that recent work \cite{cla15,gre17} on convex hulls of random walks (another measure of its shape), also has ecological relevance: to home ranges. As in that work, we take our trajectories to be Brownian paths \cite{oks10,men14}, and are able to derive closed form expressions for their asphericity. This forms an analytical first step towards understanding the relationship between the mental maps of foragers, and their search paths. 

In the context of resource foraging, B\'enichou, Redner \emph{et al}  have analytically studied  non-Markovian processes, including variations of the ``starving random walk'' \cite{beni14} when the forager is frugal \cite{beni18} and greedy \cite{bhat17}.  Models developed by ecologists have investigated foraging with memory \cite{vin15, bracis15}. The results of these simulation studies, showed that  increasing memory capacity always improved the condition (health) of the forager \cite{vin15}, but that the advantage of memory depends on the nature of the landscape \cite{bracis15}: memory is most useful when resources are high value but more difficult to encounter. Our analytical result is not specific to one particular foraging scenario, but provides a general analytical relation which may be used as a reasoning tool in cases where memory influences fitness. For example when escaping from predators, or returning to food sources, the shape of a mental map may influence the chances of success.  

\section{The model}

We consider a particle (\emph{e.g.} an animal, a forager) whose trajectory is a two-dimensional standard Brownian motion $(X(t),Y(t))_{t \geq 0}$. This animal's memory consists of a countable set of locations. We reverse the conventional direction of time so that increasing $t$ takes us further into the past, and $t=0$ is the present. Since Brownian motion is time reversal invariant, this has no affect on the statistical properties of the paths. For $0\leq t_1 < t_2$, let $A(t_1,t_2)$ be the event that the animal retains at least one of the locations it visited during the interval $[t_1,t_2]$, in particular, 
\begin{align}
\PP[A(t,t+\dt)] &= c \mu(t) \dt + o(\dt) \text{ as } \dt \ra 0,
\end{align}
where
$\mu$ is a non-increasing density function of a continuous probability distribution supported on $[0, \infty)$ with finite first moment.  Here, the function $\mu$ is referred to as the memory kernel of the forager, and the constant $c >0$ controls the overall rate at which memories are formed. In fact, $c$ gives the expected number of observations. Letting $S$ be the set of all times (excluding the current position) corresponding to the set, $L= \{ (X(s), Y(s)) \colon s \in S \}$, of locations in memory, we define the \textit{egocentric} \cite{kla98} gyration tensor
\begin{equation}
T = \begin{bmatrix} T_{11} & T_{12} \\ T_{12} & T_{22} \end{bmatrix} = \begin{bmatrix} \frac{1}{1+|S|} \sum_{t \in S}   X^2(t) & \frac{1}{1+|S|} \sum_{t \in S}  X(t)Y(t) \\
 \frac{1}{1+|S|} \sum_{t \in S}   X(t)Y(t) &  \frac{1}{1+|S|} \sum_{t \in S}   Y^2(t)
\end{bmatrix};
\label{eqn:egoTen}
\end{equation}
for technical reasons we have included the current location (at $t=0$) in the memory. Also note that $X(0)=Y(0)=0$.
The term \textit{egocentric} refers to the fact that the elements of the tensor are moments about the animal's \textit{current} location, rather than the centre of mass of all locations in memory. For simplicity, we assume that the animal's current location is the origin. Since $T$ is a positive definite  matrix, then there exists an angle $\theta \in [0,2 \pi)$ for which the rotation matrix $R(\theta)$ diagonalises $T$. We write the diagonalised matrix
\begin{equation}
\tilde{T} = R^{-1}(\theta) T R(\theta) = \begin{bmatrix}
\lambda_1 && 0 \\
0 && \lambda_2
\end{bmatrix}.
\end{equation} 
The eigenvalues $\lambda_1$ and $\lambda_2$, with $\lambda_1>\lambda_2$, give the mean square displacements of the set $L$ along the principle axes: the (orthogonal) eigenvectors of $T$. Letting $\bv{v} = \binom{x}{y}$ then the equation
\begin{equation}
\label{eqn:ellip}
\bv{v}^T  T^{-1} \bv{v}= \kappa^2
\end{equation}
describes an ellipse whose semi-major and semi-minor axes have lengths equal to $ \kappa \sqrt{\lambda_1}$ and $\kappa \sqrt{\lambda_2}$ respectively. Two examples of such ellipses are shown in Figure \ref{fig:ellipse}, along with the sets, $L$, of locations used to generate them.
\begin{figure}
	\includegraphics{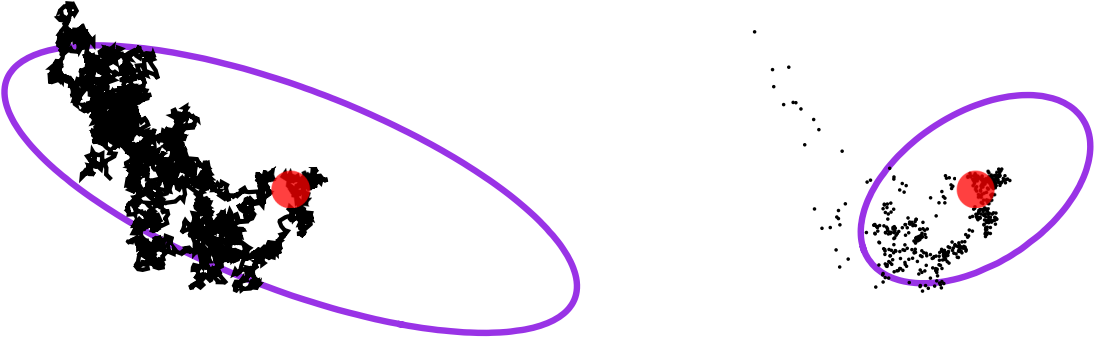}
	\caption{\label{fig:ellipse} Left: A Brownian path of duration $\tau=10$, the locations of which correspond to the memory kernel $\mu(t)=\frac{1}{\tau}\mathbf{1}_{\{t<\tau\}}$ with $c \ra \infty$. Right: The set, $L$, of locations memorised from this path when $\mu(t)=\tfrac{1}{1-e^{-\tau}}e^{-t} \mathbf{1}_{\{t<\tau\}}$ and $c=500$. The purple ellipses are given by equation (\ref{eqn:ellip}) with $\kappa=2$. The semi-transparent red dots give the current location of the walker (the origin). }
\end{figure}

In polymer physics, the shape of a molecule may be described by its \textit{standard} gyration tensor, whose elements are the moments of its mass distribution about its centre of mass \cite{rud86,rud87}. The extent to which such molecules are elongated may be measured by their asphericity, which in two dimensions is defined \cite{rud86}
\begin{equation}
A_2=\frac{\EE[(\lambda_1-\lambda_2)^2]}{\EE[(\lambda_1+\lambda_2)^2]},
\label{eqn:asp}
\end{equation}
where $\lambda_1$ and $\lambda_2$ are the eigenvalues of the gyration tensor, and the expectation is taken over all configurations of the molecule. This two dimensional asphericity is a measure of the average extent to which the ellipse defined by equation (\ref{eqn:ellip}), departs from a circle. Since a circle is just a 1-sphere we use the terms spherical and circular interchangeably. The main focus of our paper is to explore the properties of $A_2$ computed from our egocentric gyration tensor (\ref{eqn:egoTen}) where the expectation is taken over the all possible sets, $L$, of memorised locations for given memory kernel. We emphasise that the egocentric version $A_2$ has a somewhat different interpretation to the standard $A_2$. Intuitively, the egocentric asphericity measures the extent to which an ellipse centred on the current location of the walker, drawn to enclose the bulk of the walker's memory, departs from a circle. Some notable examples which illustrate the difference between the egocentric and centre--of--mass versions of $A_2$ are illustrated in Figure \ref{fig:A2Examples}. As can been seen from Figures \ref{fig:cont}, \ref{fig:ellipSE2} and \ref{fig:ellipSE02}, despite these peculiarities, the egocentric asphericity is an effective measure of the elongation of a set $L$.

\section{Analytical results}

The set of memory times, $S$, is a (non-homogeneous) Poisson point process on $\left[0, \infty \right)$ with intensity function $c\mu(t)$. In particular, $\mathbb{E}[|S|] = \int_{0}^{\infty} c\mu(t)dt = c$.  The sums defined in (\ref{eqn:egoTen}) are taken over this process. Recall that $(X(t), Y(t))_{t \geq 0}$ is a two-dimensional standard Brownian motion.  Provided the rate of memory formation is sufficiently high, we may approximate the moments of elements of $T$ by the moments of integrals with stochastic integrands, that is, for polynomials $p_1(x, y)= x^2$, $p_2(x, y)=y^2$ and $p_{3}(x, y)=xy$, as $c \to \infty$, we have
\begin{align}
\mathbb{E}\left[ \left(\frac{1}{1+|S|}\sum_{t \in S} p_i(X(t),Y(t))\right)^{k} \right] \ra \mathbb{E}\left[\left(\int_0^\infty  p_i(X(s),Y(s)) \mu(s)ds\right)^{k}\right],
\label{eqn:approx_form1}\\
\mathbb{E}\left[\frac{1}{1+|S|}\sum_{t \in S} X^2(t) \frac{1}{1+|S|}\sum_{s \in S} Y^2(s) \right]\ra \mathbb{E}\left[\int_0^\infty  X^2(s) \mu(s)ds \right]\mathbb{E}\left[\int_0^\infty  Y^2(s) \mu(s)ds \right]  
\label{eqn:approx_form2}
\end{align}
for $k \in \{1, 2\}$ and $i \in \{1, 2, 3\}$. In particular, $\int_0^\infty p_i(X(s),Y(s)) \mu(s)ds$ is well-defined since the distribution $\mu$ has finite first moment. For the proof of formulae (\ref{eqn:approx_form1}) and (\ref{eqn:approx_form2}) we refer the reader to \ref{sect:Conv}.
\subsection{Asphericity Theorem}

Let $M(t) = \int_0^t \mu(s) ds$, where $\mu$ is the memory kernel. 
Our main result is summarised in the following theorem, the proof of which requires two additional results (Lemma \ref{lemma: app.main1} and Lemma \ref{lemma: app.main2}) which are  stated and proved in \ref{sect:Lemmas}.

\begin{figure}[h!]
	\includegraphics[scale=0.4]{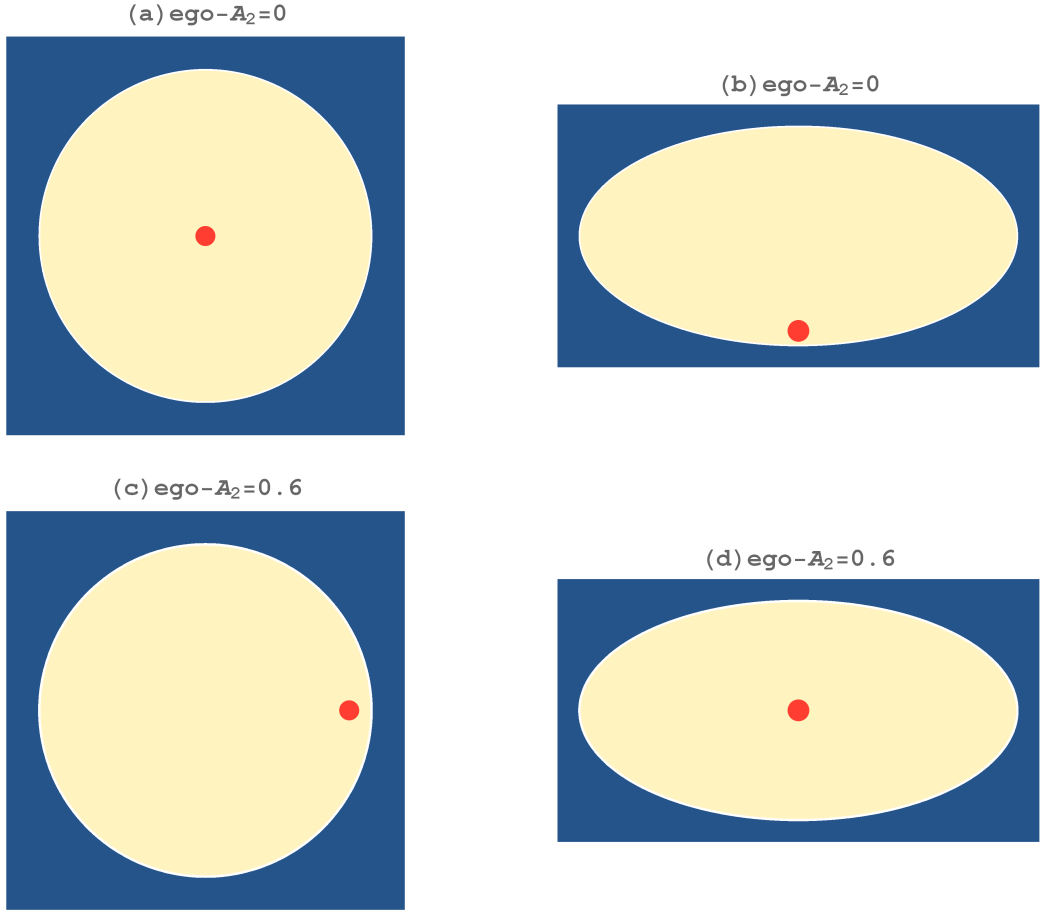}
	\caption{\label{fig:A2Examples} Examples illustrating the behaviour of egocentric asphericity. Red dots represent the current locations of walkers and yellow regions represent spatial memories. The same memory shape can have different asphericities, depending on the location of the walker. In a similar way, different memory shapes can have the same asphericity.   }
\end{figure}
\newpage
\begin{thm}[Egocentric Asphericity] \label{thm: main_thm}
	The egocentric asphericity 
	of two-dimensional Brownian motion memorised according to $\mu$, that is, $(X(t), Y(t))_{t \in S}$, as $c\to \infty$ is
	\begin{equation*}
	A_2 = 1 - 4 \lim_{\tau \to \infty} \frac{\alpha(\tau)}{\beta(\tau)},
	\end{equation*}
	where 
	\begin{align}\label{eqn: A1}
	\alpha({\tau}) = &\frac{1}{2}\left[\left(\int_0^{\tau} s \mu(s) ds\right)^2 + \left(\int_0^{\tau} M(s) ds\right)^2-\tau^2M^2({\tau})\right]  \\ \nonumber
	&+ \int_0^{\tau} \left[(4 s - \tau) M(\tau) - 2s M(s)  \right] M(s) ds, 
	\end{align}
	\begin{align}
	\label{eqn: A2}
	\beta({\tau}) = & 2\left[\left(\int_0^{\tau} s \mu(s) ds\right)^2 + \left(\int_0^{\tau} M(s) ds\right)^2 + 3\tau^2M^2({\tau})\right] \\ \nonumber
	&- 4 \int_0^{\tau} \left[(4 s + \tau) M(\tau) - 2s M(s)  \right] M(s) ds.
	\end{align}
	
\end{thm}

\begin{proof}
	Let us consider the gyration tensor $T$ in (\ref{eqn:egoTen}). 
	The asphericity $A_2$ is defined by (\ref{eqn:asp}).
	Since $\lambda_1+\lambda_2 =\mbox{Tr}(T) = T_{11} + T_{22}$,
	$ (\lambda_1 - \lambda_2)^2 = \left(\mbox{Tr}(T)\right)^2-4\lambda_1\lambda_2$ and $\lambda_1\lambda_2 = \det(T) =  T_{11}T_{22}-T_{12}^2$  we obtain
	$$A_2 =  1 - 4 \frac{\mathbb{E}\left[T_{11}T_{22}\right]-\mathbb{E}\left[T_{12}^2\right]}{\mathbb{E}\left[(T_{11}+T_{22})^2\right]} =  1 - 4 \frac{\mathbb{E}[T_{11}T_{22}] -\mathbb{E}\left[T_{12}^2\right]}{\mathbb{E}[T_{11}^2]+ 2 \mathbb{E}[T_{11}T_{22}]+\mathbb{E}[T_{22}^2]}.
	$$ 
	As $c \to \infty$, we apply formulae (\ref{eqn:approx_form1}), (\ref{eqn:approx_form2}), so $\mathbb{E}[T_{11}T_{22}]$ and $\mathbb{E}[T_{ij}^{k}]$ are equal to the corresponding expectation of time integrals with a Brownian integrand for $i \leq j \in \{1, 2\}$ and $k \in \{1, 2\}$. We may therefore replace our initial definitions of the elements $T_{ij}$ in (\ref{eqn:egoTen}), with integrals having the same moments as these elements as $c \ra \infty$, that is,
		\begin{align*}
	T_{11} = \int_0^\infty X^2(s)\mu(s)ds, \ T_{22} = \int_0^\infty Y^2(s)\mu(s)ds, \ T_{12}= \int_0^\infty X(s)Y(s)\mu(s)ds.
	\end{align*}
	
	\noindent 
	Let $\tau > 0$ and define
	\begin{align*}
	T_{11}(\tau) = \int_0^\tau X^2(s)\mu(s)ds, \ T_{22}(\tau) = \int_0^\tau Y^2(s)\mu(s)ds, \ T_{12}(\tau) = \int_0^\tau X(s)Y(s)\mu(s)ds.
	\end{align*}
	In particular, monotone and dominated convergence theorems \cite{grim01} yield
	\begin{align*}
\mathbb{E}[T_{ij}] =& \lim_{\tau \to \infty} \mathbb{E}[T_{ij}(\tau)] \mbox{ for all } i\leq j \in \{1, 2\}. 
\end{align*}
	Furthermore, independence of $X$ and $Y$, and Lemma \ref{lemma: app.main1} imply
	\begin{align*}
	\mathbb{E}[T_{11}(\tau)T_{22}(\tau)]=& \mathbb{E}[T_{11}(\tau)]\mathbb{E}[T_{22}(\tau)] = \left(\int_0^\tau s\mu(s)ds\right)^2.
	\end{align*}
	Clearly,  $\mathbb{E}[T_{11}^2(\tau)]= \mathbb{E}[T_{22}^2(\tau)]$. We then apply Lemma \ref{lemma: app.main1} and Lemma \ref{lemma: app.main2} to obtain $\mathbb{E}[T_{11}^2(\tau)]$ and $\mathbb{E}[T_{12}^2(\tau)]$, and hence the desired form of all terms of $A_2$; in particular
	\begin{align*}
	\alpha(\tau) = &\mathbb{E}\left[T_{11}(\tau)T_{22}(\tau)\right]-\mathbb{E}\left[T_{12}^2(\tau)\right], \\
	\beta(\tau) = & \mathbb{E}\left[(T_{11}(\tau)+T_{22}(\tau))^2\right]
	\end{align*}
	coincide with (\ref{eqn: A1}) and (\ref{eqn: A2}). 
	
\end{proof}

\section{Examples}

To illustrate how our egocentric asphericity theorem (Theorem \ref{thm: main_thm}) may be used to calculate asphericities, we work though four examples. A larger set of exact results are collected in Table \ref{tbl:asps}. 

\subsection{Uniform kernel }
Consider the pdf of a continuous uniform distribution with support $[0, r]$ for some $r>0$, to be a memory kernel. Then 
$\mu(t) = \frac{1}{r}\textbf{1}_{\{t < r\}}$, and the corresponding first moment equals to $\frac{r}{2}$ and $M(t) = \begin{cases} \frac{t}{r} & \mbox{if } t<r \\1& \mbox{if } t \geq r \end{cases}$. Let $r>0$ and take $\tau > r$ then
\begin{align*}
\alpha( \tau) = \frac{r^2}{12}, \
\beta(\tau) = \frac{5}{3}r^2.
\end{align*}
Clearly 
$$\lim_{\tau \to \infty} \frac{\alpha(\tau)}{\beta(\tau)} = \frac{1}{20}.$$
Hence, the corresponding asphericity is  
$$A_2 =  1 - 4 	\lim_{\tau \to \infty}\frac{\alpha(\tau)}{\beta(\tau)}  = \frac{4}{5}.$$
\subsection{Exponential kernel}
We take the memory kernel to be exponential, that is, $\mu(t) =\lambda \exp(-\lambda t)$ for  $\lambda>0$. Recall, that the first moment of exponential distribution is $\frac{1}{\lambda}$ and $M$ is the corresponding CDF, that is, $M(t) = 1-e^{-\lambda t}$. The asphericity of the memorised walk with exponential kernel is not affected by $\lambda$ so for simplicity we fix $\lambda =1$ so $\mu(t) = \exp(-t)$. Then
\begin{align*}
\alpha( \tau) =& \frac{\exp(-2 \tau)(-5 + \exp(2 \tau) - 4 \exp(\tau)(-1 + \tau) - 2 \tau)}{2 },  \\
\beta(\tau) =&  \exp(-2 \tau) (18 + 6 \exp(2 \tau) + 20 \tau + 8 \tau^2 - 
8 \exp(\tau) (3 + \tau))\\
\end{align*}
and so 
\begin{align*}
\lim_{\tau \to \infty} \frac{\alpha(\tau)}{\beta(\tau)} = \frac{1}{12}.
\end{align*}
Hence, 
$$A_2 = 1 - 4 	\lim_{\tau \to \infty}\frac{\alpha(\tau)}{\beta(\tau)} = 1 - \frac{4}{12} = \frac{2}{3}.$$

\subsection{Stretched exponential kernel}

The stretched exponential density function is defined
\begin{equation}
\mu(t)=\frac{a e^{-t^a}}{\Gamma(a^{-1})},
\end{equation}
where $\Gamma$ is the Gamma function 
\begin{equation*}
\Gamma(a) = \int_0^\infty t^{a-1} e^{-t} dt. 
\end{equation*}
The cumulative stretched exponential is
\begin{equation}
M(t) = \frac{ \gamma(a^{-1},t^a)}{\Gamma(a^{-1})}
\end{equation}
where $\gamma$ is the lower incomplete gamma function
\begin{equation*}
\gamma(a,x) = \int_0^x t^{a-1} e^{-t} dt.
\end{equation*}
The cumulative $M(t)$ is in fact a \textit{regularised} lower incomplete gamma function \cite{olv10}. 
As an example we find the ashpericity for the memorised walk with stretch exponential kernel with parameter $a= \frac{1}{2}$. 
In order to determine the limits of $\alpha(\tau)$ and $\beta(\tau)$ we note that
\begin{align*}
\int_0^\tau s \mu(s) ds =& \gamma (4, \sqrt{\tau}), \\
\int_0^{\tau} M(s) ds =&\tau - 2 \gamma(2 , \sqrt{\tau}) - 2 \gamma(3 , \sqrt{\tau}),\\
\int_0^{\tau} sM(s) ds =& \frac{1}{2}\tau^2M(\tau) - \frac{1}{2}\gamma(6, \sqrt{\tau}),\\
\int_0^{\tau} sM^2(s) ds =& \frac{1}{2}\tau^2 M^2(\tau) - \gamma(6, \sqrt{\tau}) + \frac{1}{2^6}\left(\gamma(6, 2 \sqrt{\tau}) + \frac{1}{2}\gamma(7, 2 \sqrt{\tau})\right).
\end{align*}
Taking into account that 
for $i, j \in \{1, 2\}$, $\tau^i (1- M^j(\tau)) \to 0$  as $\tau \to \infty$ and expanding the terms of $\alpha(\tau)$ and $\beta(\tau)$, we note that only terms which are $o(\tau)$, as $\tau \to \infty$, do not cancel out. Using the fact that for a continuous function $f$ such that $f(\tau) \to \infty$, as  $\tau \to \infty$, and $n \in \mathbb{N}$ we have $$\lim_{\tau \to \infty} \gamma(n, f(\tau)) = \Gamma(n)= (n-1)!,$$ we obtain 
\begin{align*}
\lim_{\tau \to \infty} \alpha(\tau) =& \lim_{\tau \to \infty}\left[2 \left(\gamma \left(2,\sqrt{t}\right)+\gamma \left(3,\sqrt{t}\right)\right)^2+\frac{\gamma \left(4,\sqrt{t}\right)^2}{2}\right.\\
& \left. +\frac{1}{32} \left(-\gamma \left(6,2
\sqrt{t}\right)-\frac{\gamma \left(7,2 \sqrt{t}\right)}{2}\right)\right]
\\
&= 21, \\
\lim_{\tau \to \infty} \beta(\tau) =& \lim_{\tau \to \infty}\left[8 \left(\gamma \left(2,\sqrt{t}\right)+\gamma \left(3,\sqrt{t}\right)\right)^2+2 \gamma \left(4,\sqrt{t}\right)^2\right.\\&\left.+\frac{1}{8} \left(\gamma \left(6,2 \sqrt{t}\right)+\frac{\gamma
	\left(7,2 \sqrt{t}\right)}{2}\right)\right]
\\
&= 204.
\end{align*}
Hence, 
$$A_2 = 1- 4 \cdot \frac{21}{204} = \frac{10}{17}.$$

\subsection{Lomax kernel}
The density function of the Lomax distribution (or the Pareto Type II distribution) is given by 
$$ \mu(t) = \frac{a}{\lambda} \left(1+ \frac{t}{\lambda}\right)^{-(a + 1)},$$
where $\lambda >0$ is a scale parameter and $a>0$ is a shape parameter. 
Similarly, to the exponential case, we fix $\lambda = 1$ since the scale would not affect the asphericity. The first moment of the Lomax distribution exists and equals to
$\frac{1}{a - 1}$ only if $a>1$; for that reason let us consider $a >1$. The corresponding CDF, $M$ is given by $M(t) = 1 - \left(1+ \frac{t}{\lambda}\right)^{-\alpha}$. 
For 
$$ \mu(t)= a \left(1 + t\right)^{-(a+ 1)}$$
we find that 
\begin{align*}
\int_0^\tau s \mu(s) ds =&\frac{ (\tau +1)^a-a \tau -1}{(a-1)(\tau +1)^{a}}, \\
\int_0^{\tau} M(s) ds =&\tau +\frac{\tau +1-(\tau +1)^a }{(a-1)(\tau +1)^{a}},\\
\int_0^{\tau} sM(s) ds =& \frac{\tau ^2}{2}-\frac{(\tau +1)^a-a \tau  (\tau +1)+\tau
	^2-1}{(a-2) (a-1)(\tau +1)^{a}},\\
\int_0^{\tau} sM^2(s) ds =&\frac{\tau^2}{2}+ \frac{(\tau +1)^{2 a}-2 a \tau  (\tau +1)+\tau ^2-1 }{2(a-1) (2 a-1)(\tau
	+1)^{2 a}}\\&-\frac{4 (\tau +1)^a-a \tau  (\tau +1)+\tau ^2-1
}{2(a-2) (a-1)	(\tau +1)^{a}}.
\end{align*}
Using the above formulae and simplifying the terms in (\ref{eqn: A1}) and (\ref{eqn: A2}) we arrive at
\begin{align*}
\lim_{\tau \to \infty} \alpha(\tau) =& \frac{a}{(a-1)^2 (2 a-1)},\\
\lim_{\tau \to \infty} \beta(\tau) = &\frac{4 (3 a-2)}{(a-1)^2 (2 a-1)}.
\end{align*}
Hence, 
\begin{align*}
A_2 = \frac{2(a-1)}{3a-2}.
\end{align*}
Now, observe that
$$\lim_{a \to 1}  \frac{2(a-1)}{3a-2} = 0.$$
Thus, as $a$ approaches $1$ the memorised walk becomes more and more spherical!

\begin{table}[h!]
	\caption{ \label{tbl:asps}
		Asphericity of memorised random walks with different memory kernels. 
	}
	\begin{center}
		\begin{tabular}{|c|c|c|} \hline 
			\textbf{distribution name} &$\mu(t)$ & $A_2$ \\
			\hline
			Uniform supported on $[0,r], \ r>0$&$\frac{1}{r}\mathbf{1}_{\{t<r\}}$ & $\frac{4}{5}=0.8$ \\\hline
			Half normal with scale equal to $1$&$\frac{\sqrt{2}}{\sqrt{\pi}} \exp\left(\frac{-t^2}{2}\right)$& $ 2- \frac{4}{\pi} \approx 0.727$\\ 	\hline
			Exponential with rate $\lambda = 1$ &$\exp(-t)$ & $\frac{2}{3} \approx 0.667$ \\ \hline
			Stretched exponential with exponent $\frac{1}{2}$ &$\frac{1}{2}\exp\left(-t^{\frac{1}{2}}\right)$ & $\frac{10}{17} \approx 0.588$ \\ \hline
			Stretched exponential with exponent $\frac{1}{4}$&$\frac{1}{24}\exp\left(-t^{\frac{1}{4}}\right)$ & $\frac{594}{1193} \approx 0.498$ \\ \hline 
			Lomax with scale $\lambda =1$ and shape $a=1.5$ &$ 1.5 \left(1 + t\right)^{-2.5}$ & $\frac{2}{5} = 0.4$ \\ \hline 
			Lomax with scale $\lambda =1$ and shape $a=1.25$ &$ 1.25 \left(1 + t\right)^{-2.25}$ & $\frac{2}{7} \approx 0.286$ \\  \hline
			Lomax with scale $\lambda =1$ and shape $a=1.05$ &$ 1.05 \left(1 + t\right)^{-2.05}$ & $\frac{2}{23} \approx 0.0870$ \\ \hline 
		\end{tabular}
		
	\end{center}
\end{table}

\section{Simulation results}

\begin{figure}[h!]
	\includegraphics{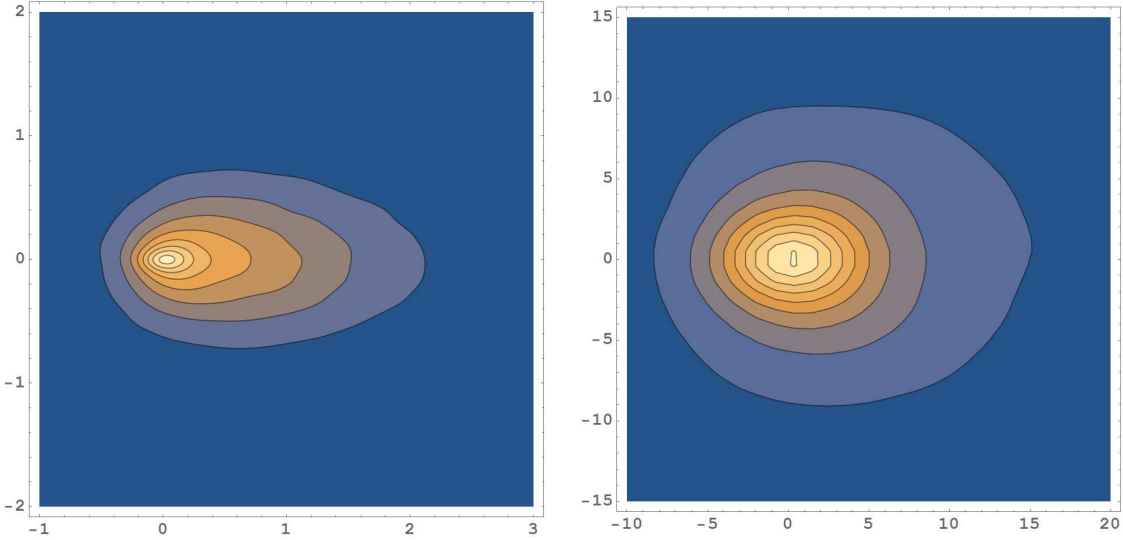}
	\caption{\label{fig:cont} Contour plots of probability density function for memorised locations when memory kernel is stretched exponential $\mu(t)=\frac{a e^{-t^a}}{\Gamma(a^{-1})}$.  Left: Kernel parameter $a=1$ corresponding to asphericity $A_2=2/3$ . Right:  Kernel parameter $a=0.5$ corresponding to asphericity $A_2=10/17$. In both cases, $c=1000$. }
\end{figure}

\begin{figure}[h!]
	\includegraphics{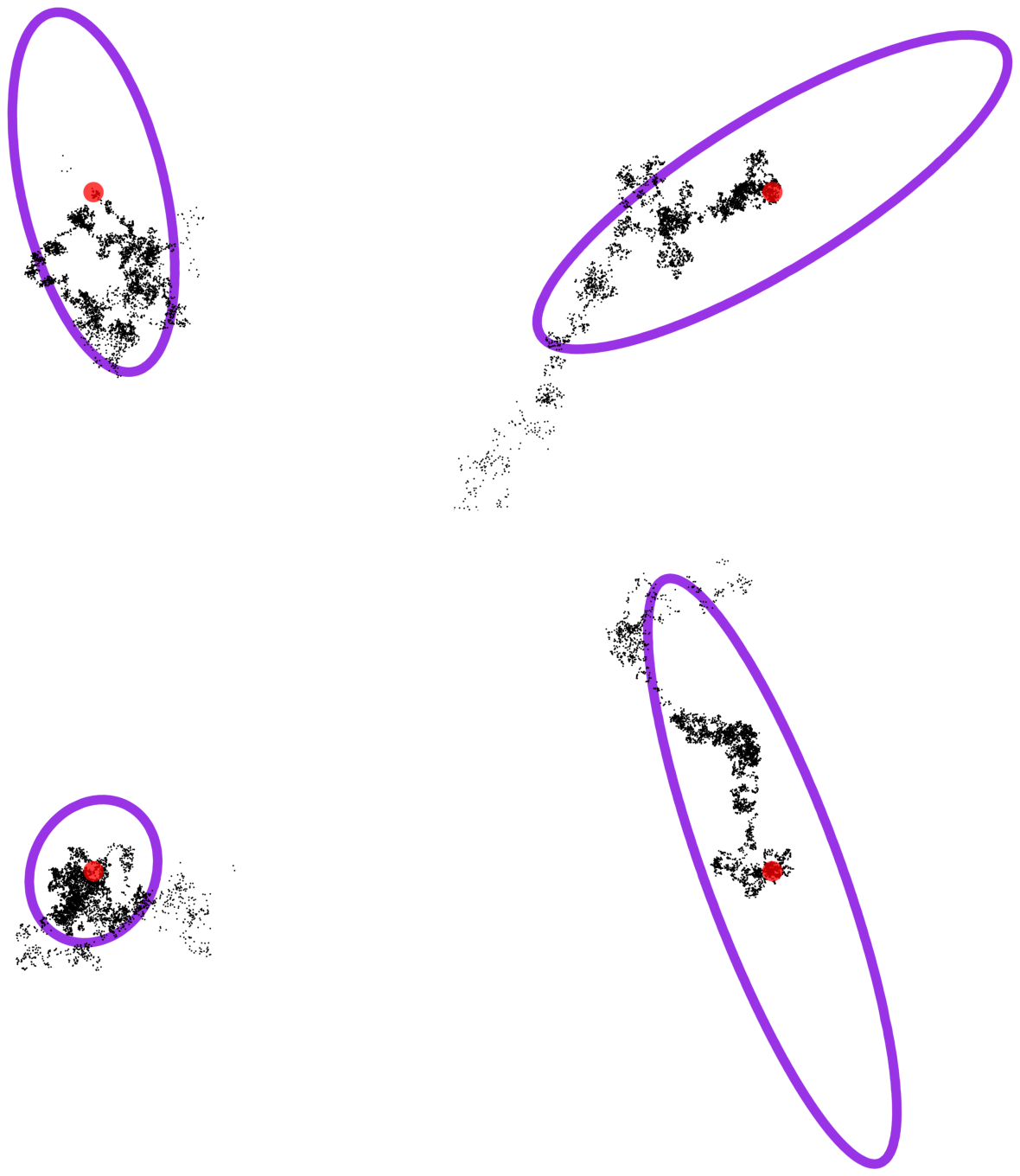}
	\caption{\label{fig:ellipSE2} Sets, $L$, of locations in a memorised random walk with stretched exponential kernel $\mu(t)=\frac{a e^{-t^a}}{\Gamma(a^{-1})}$ with parameter $a=2$ (a half-normal distribution). Memory rate $c=5000$. Purple ellipses (equation (\ref{eqn:ellip}) with $\kappa=2$) have major and minor axes aligned with eigenvectors (principle axes) of the asphericity tensor, with lengths equal to double the mean squared displacement of the memory locations along these principle directions. Red dots show current location of walker.  }
\end{figure}

\begin{figure}[h!]
	\includegraphics{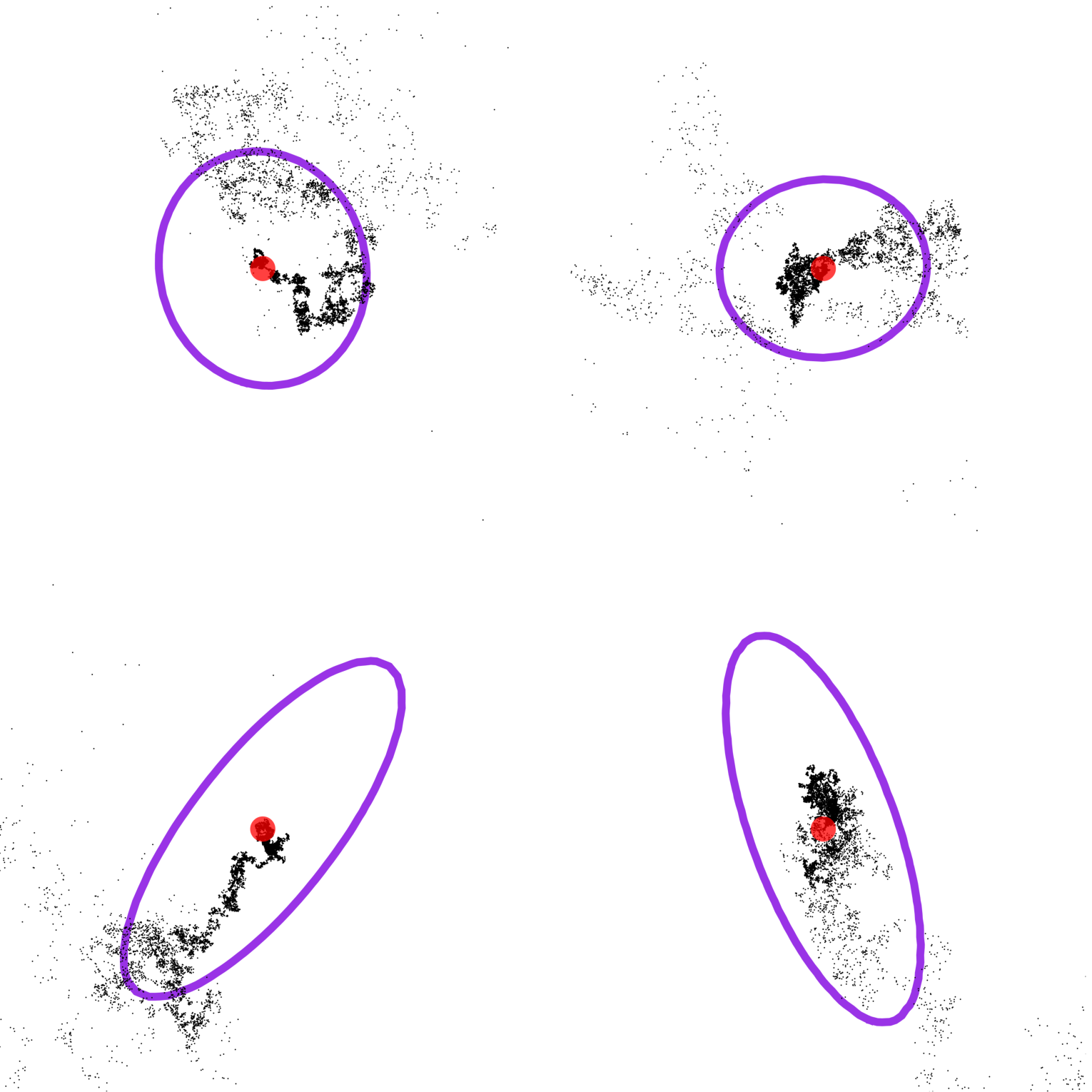}
	\caption{\label{fig:ellipSE02} Sets, $L$, of locations in a memorised random walk with stretched exponential kernel $\mu(t)=\frac{a e^{-t^a}}{\Gamma(a^{-1})}$ with parameter $a=0.2$. Memory rate $c=10^4$. Purple ellipses (equation (\ref{eqn:ellip}) with $\kappa=2$) have major and minor axes aligned with eigenvectors (principle axes) of the asphericity tensor, with lengths equal to double the mean squared displacement of the memory locations along these principle directions. Red dots show current location of walker. }
\end{figure}

To efficiently generate sample sets $S$, we make use of the fact that the set of sample times may be viewed as the arrival times of an inhomogeneous Poisson process with rate function $c \mu(t)$. At time $\tau\geq 0$, the cumulative distribution of the next arrival time $Z$ is therefore 
\begin{equation}
\PP[Z \leq t | Z>\tau] = F_\tau(t) = \begin{cases}
1- \exp \left(-c\int_{\tau}^t \mu(s) ds \right) & \text{ if } t \ge \tau \\
0 & \text{ otherwise.}
\end{cases} 
\end{equation} 
Defining 
\begin{align*}
M(t) &= \int_0^t \mu(s) ds, \\ 
p(\tau) & = \lim_{t \ra \infty} F_\tau(t)
\end{align*}
then if $u \in [0,p(\tau))$ the inverse cumulative exists and may be written as
\begin{equation}
F_\tau^{-1}(u) =  M^{-1}\left(M(\tau)-\frac{\ln(1-u)}{c}\right).
\end{equation}
Note that $p(\tau)$ is the probability that there is at least one sample (arrival) time in the interval $[\tau,\infty)$. Letting the sequence of sample times be $0<Z_1<Z_2< \ldots$, then given $Z_n$, the next sample in memory may be simulated by noting that $Z_{n+1} \stackrel{d}{=} F_{Z_n}^{-1}(U)$, where $\stackrel{d}{=}$ denotes equality in distribution, and $U$ is a random variate uniform on $[0,1]$. The result $U>p(Z_n)$ is equivalent to the event that $Z_n$ is the last sample time, and no later times are recalled. Given a simulated set, $S$, of sample times, it is straightforward to generate the set of locations $L$ by noting that
\begin{align}
X(Z_{n+1}) -  X(Z_n) & \sim \mathcal{N}(0,Z_{n+1}-Z_n) \\
Y(Z_{n+1}) -  Y(Z_n) & \sim \mathcal{N}(0,Z_{n+1}-Z_n)
\end{align}
where $\mathcal{N}(\mu,\sigma^2)$ is the normal distribution with mean $\mu$ and variance $\sigma^2$.

\subsection{Visualising the shape of the walker's memory}

We give two alternative visualisations of the shape of a walker's memory in order to provide greater intuition about the meaning of egocentric asphericity. 

\subsubsection{Averaging over many paths}

To visualise the typical shape of the walker's memory for given kernel, we construct a smooth spatial density function using a large number of memory sets $\{L_1, L_2, \ldots \}$. The ensemble of all points in these sets forms a spherically symmetric distribution, which fails to capture the aspherical shape of individual sets. In order to capture this asphericity within our average, before averaging we individually rotate each set so that the major axis of its inertia tensor is horizontally aligned and its centre of mass has positive horizontal coordinate. This latter condition allows us to see the extent to which the walker is displaced from the centre of mass of its memory. Two examples of the resulting density, computed using kernel density estimation tools available in Mathematica\textsuperscript{\textregistered},  are illustrated in Figure \ref{fig:cont}. We see that the heavier tailed memory kernel results in a more spherically symmetric memory distribution, as predicted by our asphericity calculation. We note that the centre of mass condition yields a tear-drop shaped distribution, reflecting the fact that the memory of a walker, as well as being elongated, has a centre of mass which is displaced from his current position. We note that without the centre of mass conditions the densities in Figure \ref{fig:cont} would be approximately ellipsoidal in shape rather than tear-drops.  We emphasizs that our asphericity calculations explore only the average elongation of a ellipses (not tear drops) derived from the inertia tensor of each set $L$. We now illustrate some individual examples of such ellipses.

\subsubsection{Examining individual paths}

We now explore how egocentric gyration tensor, and its corresponding ellipses, are related to the shapes of individual sets $L$. Figures \ref{fig:ellipSE2} and \ref{fig:ellipSE02} show two sets of walks, each with a superimposed ellipse (equation (\ref{eqn:ellip}) with $\kappa=2$) having major and minor axes aligned with eigenvectors (principle axes) of the asphericity tensor, with lengths equal to double the mean squared displacements of the memory locations along these principle directions. From Figures \ref{fig:ellipSE2} and \ref{fig:ellipSE02} we see that ellipses defined in this way enclose the majority of the memory locations, and provide a simple characterisation of the shape of the set $S$. Asphericity is the average extent to which such ellipses depart from circular form. Memory kernels which are more rapidly decaying tend to produce more elongated memory sets, as in Figures \ref{fig:ellipSE2}. In Figure \ref{fig:ellipSE02}, the kernel has a much fatter tail, producing a memory set which is typically more evenly distributed about the current location. We emphasise that memory sets generated using a given memory kernel may a have a wide range of elongations, so only the asphericity may be used to draw conclusions about average elongation: Figures \ref{fig:ellipSE2} and \ref{fig:ellipSE02} are illustrative examples.

\begin{figure}[h!]
	\centering
	\includegraphics[width=0.75 \textwidth]{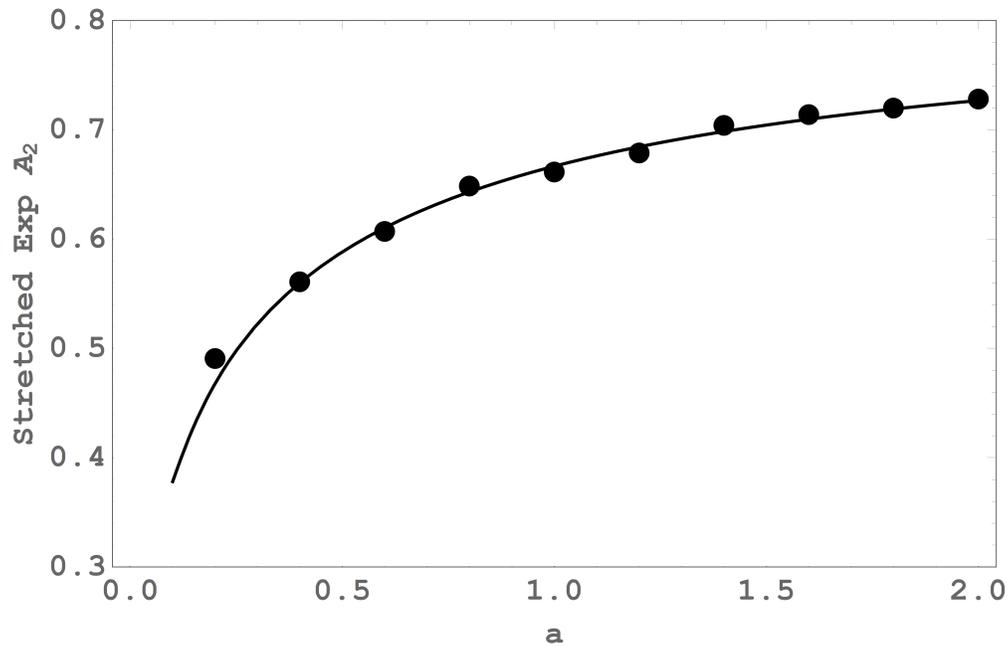}
	\caption{\label{fig:A2_SE} Asphericity of memorised random walk with stretched exponential kernel $\mu(t)=\frac{a e^{-t^a}}{\Gamma(a^{-1})}$ as a function of kernel parameter $a$. Dots show simulated asphericities when $c=2000$, computed from $10^4$ simulations.   }
\end{figure}

\begin{figure}[h!]
	\centering
	\includegraphics[width=0.75 \textwidth]{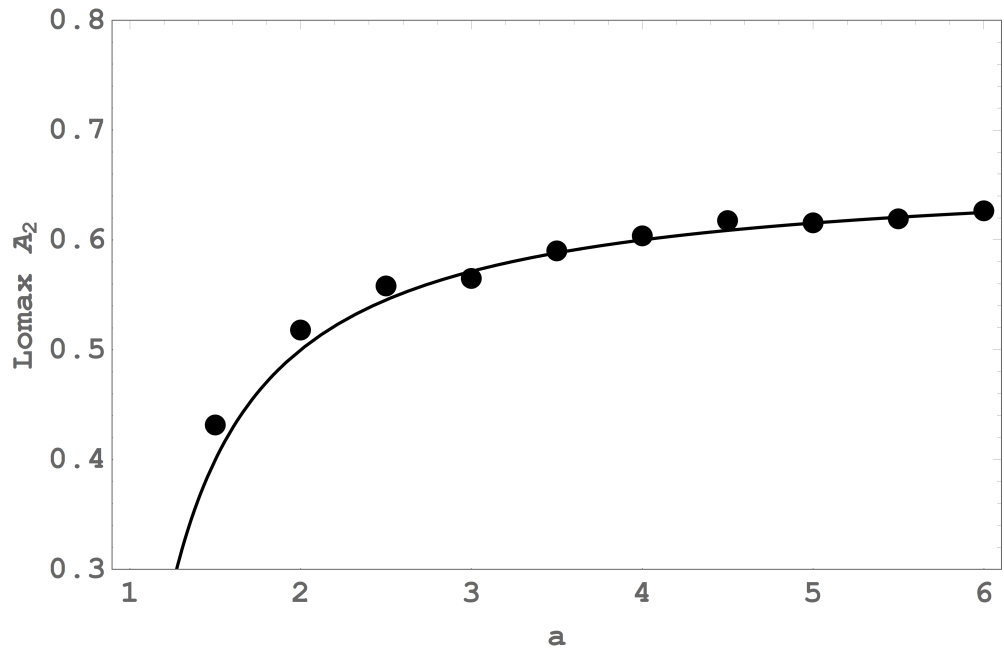}
	\caption{\label{fig:A2_LO} Asphericity of memorised random walk with Lomax kernel $\mu(t) = a \left(1+ t\right)^{-(a + 1)}$ as a function of kernel parameter $a$. Circles show simulated asphericities when $c=10^4$, computed from $10^4$ simulations. }
\end{figure}

\begin{figure}[h!]
	\centering
	\includegraphics[width=0.75 \textwidth]{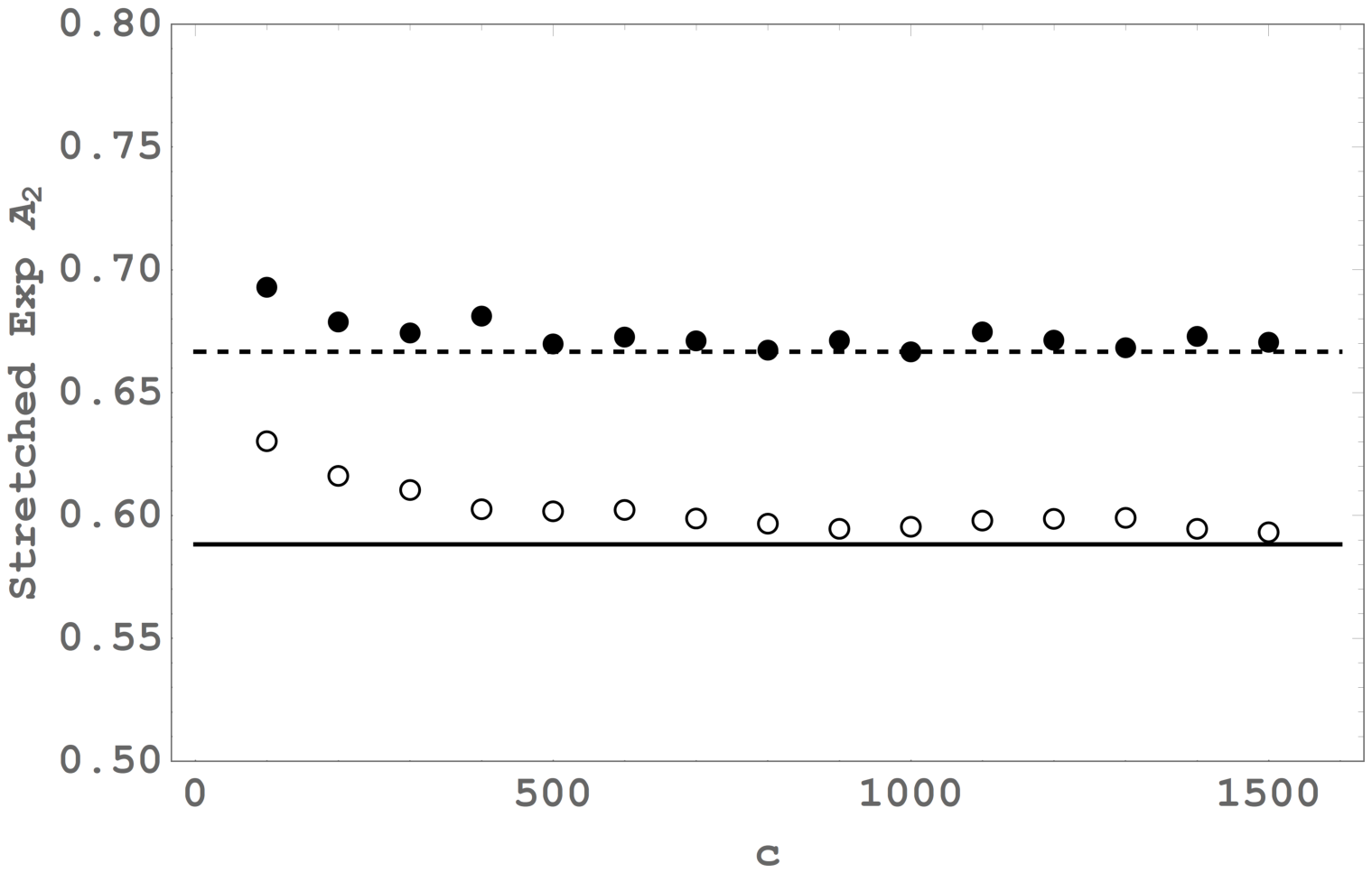}
	\caption{\label{fig:A2_c} Dependence of asphericity on the rate parameter $c$. Lines give asphericities in the $c \ra \infty$ limit for two forms of the stretched exponential distribution: $a=1$ (Dashed) and $a=1/2$ (solid). Dots and circles give simulated asphericities estimated from $10^4$ simulations.  }
\end{figure}

\subsection{Comparison between simulation and analytical results}

To verify our asphericity theorem we estimate $A_2$ by generating large numbers of sets $L$ for given memory kernel and rate $c$, computing the egocentric gyration tensor for each set, and then averaging the sum and difference of its eigenvalues over all sets. In Figure \ref{fig:A2_SE} we consider the stretched exponential kernel, and compare the results of these simulations to our theorem. Figure \ref{fig:A2_LO} gives the corresponding results for the Lomax kernel. In both cases we see a close match between simulation and analysis. We also note that asphericity declines as the tail of the kernel becomes fatter. An intuitive explanation of this effect is that for fat tailed memory kernels the time intervals between observations grow rapidly, and therefore so do the distances been successive memory locations. As a result the spatial memory is more diffuse, and space is covered more evenly with locations. For more sharply decaying kernels, each successive location is close to the last, producing a set $L$ with a structure similar to that of the original walk which was used to create it. 

The simulated asphericity results for the Lomax kernel (Figure \ref{fig:A2_LO}) show some small but systematic deviation from our analytical predictions as $a \ra 1$. This reflects the fact that for finite $c$, the set $S$ is finite ($\EE[|S|]=c$), and therefore cannot fully capture the large $t$ behaviour of $\mu(t)$. The effect becomes more pronounced for smaller $c$ values and for fatter tailed kernels. To illustrate this, in Figure \ref{fig:A2_c} we have examined the dependence of the simulated $A_2$ value on the rate parameter $c$. We see that for small $c$ values, simulations overestimate the asphericity because the tail of the kernel is under--sampled. As $c$ increases the simulated $A_2$ approaches the limiting analytical value. This rate of convergence is slower for the fatter tailed kernel, because a greater number of memory times are required to effectively sample the long tail.

\section{Discussion}

We have studied a very simple model of how a foraging animal might construct a mental map \cite{nor09,jan98} of its surroundings, showing that its memory kernel has a significant effect on the spatial structure of recalled locations. We have provided an analytical expression giving an egocentric version of the asphericity \cite{rud86} $A_2$ of the set, $L$, of memory locations in terms of the memory kernel $\mu$. We have verified our analytical result by simulation, and illustrated how $A_2$ describes the shape of the set $L$, and how this shape is related to the tail behaviour of $\mu$. Specifically, fatter tailed kernels produce more spherical memory sets.

	It has been suggested that future research into animal movement should explore the effects of memory \cite{mor04, patt08}. Many creatures utilise memory in order to efficiently exploit their environments \cite{duk99,rot10,nor09,avg13}, and memory effects are known to effect the nature of foraging or search paths \cite{fag13,fag14,boy12,boy14,oha05}.
Field studies confirm that spatial memory and learning are a critical part of foraging patterns of primates \cite{gar89}. Of particular relevance to the current work, is the fact that the movement of bearded saki monkeys is more consistent with Brownian motion than L\'{e}vy walk \cite{sha14}.
 By selecting this simplest possible path model (Brownian motion), we have been able to show how these paths combine with memory to create a spatial map of the creature's surroundings. The possession of a memory comes with physiological costs \cite{duk99,fag13}, as does searching for food, so one would expect evolution to simultaneously optimise both paths, and the use of memory. The optimal shape and structure of a mental map will depend on the environment and objectives of the forager, so given this optimal shape, one might expect the trajectories to be influenced in part by the need to create an effective map. We therefore suggest that the search for an understanding of the relationship between trajectories, maps, and memory is worthwhile. In particular, finding the asphericity from the field data of, \emph{e.g.}, brearded saki monkeys or other primates forging patterns and fitting the memory kernel from our model would improve the current understanding of their spatial memory.
	Elaborated discussion on impact and scientific importance of mathematical modelling of movement of the foragers has been explored in \cite[Part I, section 1.3]{vis11}.

We have used ellipses as simple approximations to the shapes of a memorised sets, but the ellipse is the simplest deformation of a circle. A natural extension is to consider higher order deformations which more accurately characterise the shape of $L$. Another possibility, which will be the subject of future research, is to consider L\'{e}vy walks and anomalous diffusions, \emph{e.g.} fractional Brownian motion with Hurst parameter $H \in (0, 1) \setminus \{\frac{1}{2} \}$ \cite{bia08}. 

\ack

\addcontentsline{toc}{section}{Acknowledgments}%

The authors thank to anonymous referees for insightful remarks.
The authors thank Thomas Kecker for bringing to their attention Lemma \ref{lem: lim}.
\noindent
The last named author (J.B.) was supported by a Leverhulme Trust Research Fellowship (RF-2016-177) while developing the ideas in this work, and is very grateful for this.

\noindent
The second named author (A.A.) is supported by a PhD scholarship from King Abdulaziz University, Saudi Arabia, and is very grateful for this.

\begin{appendices}
\section{Limit Formulae}
\label{sect:Conv}

In this section we elaborate on random sums over a Poisson process so that we can prove formulae \ref{eqn:approx_form1} and \ref{eqn:approx_form2}.
 \renewcommand{\thelemma}{\Alph{section}.\arabic{lemma}}
\noindent First let us emphasise a simple observation.
\begin{lemma}\label{lem: lim}
	Let $c >0$, for any $N \in \mathbb{N}$ we have 
\setcounter{equation}{0}
\renewcommand\theequation{A.\arabic{equation}}
	\begin{equation}
	\lim_{c \to \infty} c^N e^{-c} \sum_{k=0}^{\infty} \frac{c^k}{k! (k+1)^N} = 1.
	\end{equation}
\end{lemma}
\noindent  For the rest of this section let us denote 
\renewcommand\theequation{A.\arabic{equation}}
\begin{equation}I(f, n, \lambda) = \int_0^{\infty} f^n(t) \lambda(t) dt \end{equation}
for a measurable  $f \colon [0, \infty) \rightarrow \mathbb{R}$, $n \in \mathbb{N}$ and $\lambda$ such that $\int_0^{\infty} \lambda(t) dt <\infty$, provided that this integral converges. We also use the following notation for the powers of $I(f, n, \lambda)$, for $k \in \mathbb{N}$ we denote
$$ I^k(f, n, \lambda) := \left(I(f, n, \lambda) \right)^k.$$
 \renewcommand{\thethm}{\Alph{section}.\arabic{thm}}
 \setcounter{thm}{1}
\begin{thm}[Campbell's theorem \cite{str10, king02}]
	Let $S$ be a non-homegenous Poisson point process on $[0, \infty)$ with intensity function $\lambda(t)$.  Then for any measurable $f \colon [0, \infty) \rightarrow \mathbb{R}$ such that
	\begin{equation}\label{eqn:charcond}\int_0^{\infty} \min\{1, |f(t)|\} \lambda(t) dt <\infty \end{equation}
	the characteristic function of 
	$\sum_{t \in S} f(t)$ is given by 
	$$	\varphi(\theta) = \mathbb{E}\left[\exp\left( i \theta \sum_{t \in S} f(t) \right)\right] =
	\exp\left(\int_0^{\infty} \left( e^{i\theta f(t)}-1\right)\lambda(t)dt\right)$$	
provided that the integral on the right-hand side of the formula exists. 
\end{thm}
\noindent Campbell's theorem has the following simple consequences. 
 \renewcommand{\thecor}{\Alph{section}.\arabic{cor}}
\begin{cor}
	\label{cor:cam}
	Let $f$, $g \colon [0, \infty) \rightarrow \mathbb{R}$ be measurable functions such that (\ref{eqn:charcond}) holds, then 
\begin{align*}
\mathbb{E}\left[ \sum_{t \in S} f(t) \right] =&I(f, 1, \lambda), \\
\mathbb{E}\left[\left( \sum_{t \in S} f(t)\right)^2 \right] =& I^2(f, 1, \lambda) + I(f, 2, \lambda),   \\
\mathbb{E}\left[\left( \sum_{t \in S} f(t)\right)^4 \right] =& I^4(f, 1, \lambda) + 6 I^2(f, 1, \lambda)I(f, 2, \lambda)  +3I^2(f, 2, \lambda)\\& + 4 I(f, 1, \lambda)I(f, 3, \lambda)+ I(f, 4, \lambda) ,  \\
\mathbb{E}\left[ \sum_{t \in S} f(t)  \sum_{s \in S} g(s)\right] =& I(f, 1, \lambda)I(g, 1, \lambda) + I(f\cdot g, 1, \lambda). 
\end{align*}
\end{cor}
\noindent For the rest of this section we assume that $S$ is a Poisson point process with intensity $\lambda(t) = c \mu(t)$, where $c>0$ and $\mu$ is a memory kernel, that is, a non-increasing density function of a continuous probability distribution supported on $[0, \infty)$ with finite first moment. By considering powers of $1/c$ in Corollary \ref{cor:cam}, we have the following result. 
\begin{cor} \label{cor: lim1}
	Let $f$, $g \colon [0, \infty) \rightarrow \mathbb{R}$ be measurable functions  satisfying (\ref{eqn:charcond}) holds, then 
	\begin{align*}
	\lim_{c \to \infty}\mathbb{E}\left[ \frac{1}{c}\sum_{t \in S} f(t) \right] =&I(f, 1, \mu), \\
	\lim_{c \to \infty}\mathbb{E}\left[\left( \frac{1}{c} \sum_{t \in S} f(t)\right)^2 \right] =& I^2(f, 1, \mu),   \\
	\lim_{c \to \infty}\mathbb{E}\left[\left( \frac{1}{c} \sum_{t \in S} f(t)\right)^4 \right] =& I^4(f, 1, \mu),  \\
	\lim_{c \to \infty}\mathbb{E}\left[ \frac{1}{c} \sum_{t \in S} f(t)   \frac{1}{c}\sum_{s \in S} g(s)\right] =& I(f, 1, \mu)I(g, 1, \mu). 
	\end{align*}
\end{cor}
\noindent Our next corollary is required to prove formulae (\ref{eqn:approx_form1}) and (\ref{eqn:approx_form2}). 
\begin{cor} 
	Let $f$, $g \colon [0, \infty) \rightarrow \mathbb{R}$ be measurable functions satisfying (\ref{eqn:charcond}) holds, then 
	\label{cor: lim2}
	\begin{enumerate}
		\item[i)]
		$$\lim_{c \to \infty }\mathbb{E}\left[\frac{1}{1+|S|} \sum_{t \in S} f(t)\right] =I(f, 1, \mu),$$
		\item[ii)] $$\lim_{c \to \infty }\mathbb{E}\left[\left(\frac{1}{1+|S|} \sum_{t \in S} f(t)\right)^2\right] =  I^2(f, 1, \mu),$$
		\item[iii)] $$\lim_{c \to \infty }\mathbb{E}\left[\frac{1}{1+|S|} \sum_{t \in S} f(t)\frac{1}{1+|S|} \sum_{s \in S} g(s) \right] = I(f, 1, \mu)I(g, 1, \mu).$$
	\end{enumerate}
\end{cor}
\begin{proof}
	
\begin{itemize}
	\item[i)] Noting Corollary \ref{cor: lim1}, it is enough to show that 
	$$\mathbb{E}\left[ \frac{1}{1+|S|} \sum_{t \in S} f(t) \right] - \mathbb{E}\left[ \frac{1}{c}\sum_{t \in S} f(t) \right] \stackrel{c \to \infty}{\rightarrow } 0.$$

To show this, we apply the Cauchy--Schwarz inequality \cite{grim01}, we obtain
	\begin{align*}
	\left|\mathbb{E}\left[ \frac{1}{1+|S|} \sum_{t \in S} f(t) \right] - \mathbb{E}\left[ \frac{1}{c}\sum_{t \in S} f(t) \right]\right| \leq& 
	\mathbb{E}\left[  \left| \left(\frac{c}{1+|S|} - 1\right)\right|\frac{1}{c}\sum_{t \in S} |f(t)|\right]\\\leq& 
	\left(\mathbb{E}\left[  \left(\frac{c}{1+|S|} - 1\right)^2\right] \right)^{\frac{1}{2}} \\&\cdot \left(\mathbb{E}\left[\left(\frac{1}{c}\sum_{t \in S} |f(t)|\right)^2 \right]\right)^{\frac{1}{2}}.
	\end{align*}
	Corollary  \ref{cor: lim1} implies that the second term of the product converges to $I(|f|, 1, \mu)$, and so it is enough to show that the first term converges to $0$. 
By employing the law of total expectation and Lemma \ref{lem: lim} we obtain  
\begin{align*}
\mathbb{E}\left[ \left(\frac{c}{1+|S|} - 1\right)^2\right] =& \mathbb{E}\left[  \frac{c^2}{(1+|S|)^2} -  \frac{2c}{1+|S|} + 1\right]\\
=& 1+ c^2\sum_{k =0}^{\infty} \frac{1}{(1+k)^2}\mathbb{P}(|S|=k) - 
2 c \sum_{k =0}^{\infty} \frac{1}{1+k}\mathbb{P}(|S|=k) \\
=& 1 + c^2 e^{-c} \sum_{k=0}^{\infty} \frac{c^{k}}{k!(k+1)^2} - 2ce^{-c} \sum_{k=0} \frac{c^{k}}{k!(k+1)}\\
\stackrel{c \to \infty}{\to}& 0.
\end{align*}
\item[ii)] Similarly to i) we have
\begin{align*}
&\left|\mathbb{E}\left[ \left(\frac{1}{1+|S|} \sum_{t \in S} f(t)\right)^2 \right] - \mathbb{E}\left[ \left(\frac{1}{c}\sum_{t \in S} f(t)\right)^2 \right]\right|\\ \leq &
\mathbb{E}\left[  \left| \left(\frac{c^2}{(1+|S|)^2} - 1\right)\right|\frac{1}{c^2}\left(\sum_{t \in S} f(t)\right)^2\right]\\\leq &
\left(\mathbb{E}\left[ \left(\frac{c^2}{(1+|S|)^2} - 1\right)^2\right] \right)^{\frac{1}{2}} \left(\mathbb{E}\left[\left(\frac{1}{c}\sum_{t \in S} f(t)\right)^4 \right]\right)^{\frac{1}{2}}.
\end{align*}
The second term of the product converges to $I^2(f, 1, \mu)$, so we show (using Lemma \ref{lem: lim}) that the first term converges to $0$,  
\begin{align*}
\mathbb{E}\left[ \left(\frac{c^2}{(1+|S|)^2} - 1\right)^2\right] =& 
1 - 2c^2 \mathbb{E}\left[ \frac{1}{(1+|S|)^2} \right] + c^4 \mathbb{E}\left[\frac{1}{(1+|S|)^4} \right] \\
=& 1 - 2c^2e^{-c} \sum_{k=0}^{\infty} \frac{c^k}{k!(k+1)^2} + c^4e^{-c} \sum_{k=0}^{\infty} \frac{c^k}{k!(k+1)^4}\\
\stackrel{c \to \infty}{\to}& 0.
\end{align*}
\item[iii)] The third part of Corollary \ref{cor: lim2} is proved in a similar way. We have
\begin{align*}
\left|\mathbb{E}\left[\frac{1}{1+|S|} \sum_{t \in S} f(t)\frac{1}{1+|S|} \sum_{s \in S} g(s) \right]- \mathbb{E}\left[\frac{1}{c} \sum_{t \in S} f(t)\frac{1}{c} \sum_{s \in S} g(s) \right]\right|  \\\leq 
\mathbb{E} \left[ \left|\left(\frac{c^2}{(1 +|S|)^2} - 1 \right)\right|   \frac{1}{c} \sum_{t \in S} |f(t)| \frac{1}{c} \sum_{s \in S} |g(s)| \right] \\ \leq
\left(\mathbb{E}\left[ \left(\frac{c^2}{(1 +|S|)^2} - 1 \right)^2  \frac{1}{c^2} \left(\sum_{t \in S} |f(t)|\right)^2 \right] \right)^{\frac{1}{2}} \left(\mathbb{E}\left[ \left(\frac{1}{c} \sum_{s \in S} |g(s)| \right)^2 \right] \right)^{\frac{1}{2}}.
\end{align*}
The last term of the product clearly converges to $I(|g|, 1, \mu)$, thus it is enough to show that the square of first term converges to $0$. We apply, Cauchy--Schwartz inequality again, to get
\begin{align*}
\mathbb{E}\left[ \left(\frac{c^2}{(1 +|S|)^2} - 1 \right)^2  \frac{1}{c^2} \left(\sum_{t \in S} |f(t)|\right)^2 \right] \leq& \left(
\mathbb{E}\left[ \left(\frac{c^2}{(1 +|S|)^2} - 1 \right)^4 \right]\right)^{\frac{1}{2}}\\&\cdot \left( \mathbb{E}\left[   \left(\frac{1}{c}\sum_{t \in S} |f(t)|\right)^4 \right]\right)^{\frac{1}{2}}.
\end{align*}
Here, again, the second term converges to $I^2(|f|, 1, \mu)$, so we show that the first term converges to $0$, 
\begin{align*}
&\mathbb{E}\left[ \left(\frac{c^2}{(1 +|S|)^2} - 1 \right)^4 \right] \\=&
\mathbb{E}\left[ \frac{c^8}{(1+|S|)^8} - 4 \frac{c^6}{(1+|S|)^6}+6 \frac{c^4}{(1+|S|)^4}-4 \frac{c^2}{(1+|S|)^2} + 1 \right]
\\=& 1 + c^8e^{-c}\sum_{k\geq 0} \frac{c^k}{k!(k+1)^8}-4c^6e^{-c}\sum_{k\geq 0} \frac{c^k}{k!(k+1)^6} + 6 c^4 e^{-c}\sum_{k\geq 0} \frac{c^k}{k!(k+1)^4} \\&- 4c^2e^{-c}\sum_{k\geq 0} \frac{c^k}{k!(k+1)^2} \\
\stackrel{c \to \infty}{\to}& 0.
\end{align*}
\end{itemize}
\end{proof}
\noindent Our main result in this section is a consequence following from Corollary \ref{cor: lim2}. 
\begin{cor}
	Let $(X(t), Y(t))_{t \geq 0}$ be a two-dimensional standard Brownian motion. 
	For polynomials $p_1(x, y)= x^2$, $p_2(x, y)=y^2$ and $p_{3}(x, y)=xy$, as $c \to \infty$, we have
	\begin{align}
	\mathbb{E}\left[ \left(\frac{1}{1+|S|}\sum_{t \in S} p_i(X(t),Y(t))\right)^{k} \right] \ra \mathbb{E}\left[\left(\int_0^\infty  p_i(X(s),Y(s)) \mu(s)ds\right)^{k}\right],
	\\
	\mathbb{E}\left[\frac{1}{1+|S|}\sum_{t \in S} X^2(t) \frac{1}{1+|S|}\sum_{s \in S} Y^2(s) \right]\ra \mathbb{E}\left[\int_0^\infty  X^2(s) \mu(s)ds \right]\mathbb{E}\left[\int_0^\infty  Y^2(s) \mu(s)ds \right]  
	\end{align}
	for $k \in \{1, 2\}$ and $i \in \{1, 2, 3\}$. In particular, $\int_0^\infty p_i(X(s),Y(s)) \mu(s)ds$ is well-defined since the distribution $\mu$ has finite first moment.
\end{cor}
\begin{proof}
	This result can be obtained by applying the tower property \cite{grim01} of the expectation by conditioning on the Brownian paths and applying the techniques used in the proof of Corollary \ref{cor: lim2}.
\end{proof}

\noindent  For more details on the properties of moments of random sum  over a Poisson process we refer the reader to \cite{str10, king02, bacc02}.

\section{Additional Lemmas}
\label{sect:Lemmas}

Let $\mu$ be a memory kernel, we denote its integral by $M(t) = \int_0^t \mu(s) ds$ for all $t \geq 0$. 

\noindent In order to complete the proof of Theorem \ref{thm: main_thm} we require a series of lemmas; these lemmas use standard techniques of It\^{o} stochastic calculus \cite{oks10}.
 \renewcommand{\thelemma}{\Alph{section}.\arabic{lemma}}
\setcounter{lemma}{0}  
\begin{lemma}\label{lem: app.square}
	Let $W$ be a one-dimensional standard Brownian motion then
	\begin{align*}
	\int_{0}^{t} W^2(s)\mu(s)ds = M(t)W^2(t) - 2 \int_0^t  M(s)W(s) dW(s) - \int_0^t M(s)ds \ \quad (t \geq 0).
	\end{align*} 
\end{lemma}
\begin{proof}
	Use It\^{o}'s lemma
	with $f(t,s)=M(t)s^2$. 
\end{proof}
\begin{lemma} \label{lemma: app.main1}
	Let $W$ be a one-dimensional standard Brownian motion. The first two moments of $\int_0^t  W^2(s)\mu(s)ds$ for all $t \geq 0$ are
	\begin{align*}
	\mathbb{E}\left[\int_{0}^{t} W^2(s)\mu(s)ds  \right] =&
	\int_0^t s\mu(s)ds,\\
	\mathbb{E}\left[\left(\int_{0}^{t} W^2(s)\mu(s)ds\right)^2  \right] =&
	3t^2M^2(t) + 4 \int_0^t s M^2(s)ds + \left( \int_0^t M(s)ds\right)^2 \\&- 8 M(t)\int_0^t sM(s) ds - 2tM(t)\int_0^t M(s) ds.
	\end{align*}
\end{lemma}

\begin{proof}
	
	The formula for the expectation is a simple consequence of Fubini's theorem. For the second moment use Lemma \ref{lem: app.square}, the fact that $W^2(t) = t+ 2 \int_0^t W(s)dW(s)$ and It\^{o}'s isometry. 
\end{proof}
\begin{lemma} \label{lemma: app.A.3}
	Let $W_1$ and $W_2$ be two independent one-dimensional standard Brownian motions. Then 
	\begin{align*}
	\int_0^t W_1(s)W_2(s)  \mu(s)ds = \int_0^t  B^2(s) \mu(s)ds - \frac{1}{2}\int_0^t  W_1^2(s)\mu(s) ds -\frac{1}{2}\int_0^t W_2^2(s)\mu(s) ds,
	\end{align*}
	where $B$ is a standard Brownian motion given by $B(t) = \frac{W_1(t)+W_2(t)}{\sqrt{2}}$. 
	In particular, the correlation coefficient of $B$ and $W_i$ for $i \in \{1,2\}$ is 
	$\rho_{B(t), W_1(t)} = \rho_{B(t), W_2(t)} = \frac{1}{\sqrt{2}}.$
\end{lemma}
\begin{proof}
	Straightforward verification.
\end{proof}
\begin{lemma} \label{lemma: app.A.4}
	Let $W$ and $B$ be two correlated Brownian motions with correlation coefficients $\rho \in [-1,1]$ then 
	\begin{align}\nonumber \mathbb{E}\left[\int_0^tB^2(s)  \mu(s)  ds\int_0^t W^2(s)  \mu(s) ds \right] =& \left(\int_0^t M(s)ds \right)^2 -2tM(t)\int_0^t M(s)ds \\\nonumber&+ (1 +2\rho^2) t^2M^2(t)
	-8\rho^2  M(t)\int_0^t sM(s) ds \\&+ 4\rho^2 \int_0^t s M^2(s) ds.  \label{eqn: app.first} \end{align}
\end{lemma}
\begin{proof}
	To calculate $\mathbb{E}\left[\int_0^t B^2(s)  \mu(s) ds\int_0^t  W^2(s)  \mu(s) ds \right]$ we employ Lemma \ref{lem: app.square}. Thus, we obtain
	\begin{align*}
	&\mathbb{E}\left[\int_0^t B^2(s)  \mu(s)ds\int_0^t W^2(s) \mu(s) ds \right]  \\=&\mathbb{E}\left[\left(M(t)B^2(t) - 2 \int_0^t  M(s)B(s) dB(s) - \int_0^t M(s)ds \right)\right.\\
	&\left.\left( M(t)W^2(t) - 2 \int_0^t  M(s)W(s) dW(s) - \int_0^t M(s)ds\right)\right]\\
	=& M^2(t)\mathbb{E}\left[B^2(t)W^2(t) \right]-2M(t)\mathbb{E}\left[B^2(t)  \int_0^t  M(s)W(s) dW(s)\right]\\
	& - M(t)\int_0^t M(s)ds\mathbb{E}[B^2(t)] - 2 M(t) \mathbb{E}\left[W^2(t)\int_0^t  M(s)B(s) dB(s)\right] \\
	=&\left(\int_0^t M(s)ds \right)^2 -2tM(t)\int_0^t M(s)ds + M^2(t)\mathbb{E}\left[B^2(t)W^2(t) \right]\\
	&-2M(t)\mathbb{E}\left[B^2(t)  \int_0^t  M(s)W(s) dW(s)\right]\\&- 2 M(t) \mathbb{E}\left[W^2(t)\int_0^t  M(s)B(s) dB(s)\right] \\
	&+4 \mathbb{E}\left[\int_0^t  M(s)B(s) dB(s)\int_0^t  M(s)W(s) dW(s)  \right].
	\end{align*}
	The first expectation term is 
	\begin{align*}
	\mathbb{E}\left[B^2(t)W^2(t)\right] =&
	\mathbb{E}\left[ \left(t+ 2 \int_0^t B(s)dB(s)\right)\left( t+ 2 \int_0^t W(s)dW(s)\right) \right] \\=& t^2 + 4\rho\int_0^t \mathbb{E}[B(s)W(s)]ds\\=&
	(1 +2\rho^2) t^2.
	\end{align*}
	In order to determine the other expectation terms we use It\^{o}'s product rule and Fubini's theorem. Thus,
	\begin{align*}
	\mathbb{E}\left[\int_0^t  M(s)B(s) dB(s)\int_0^t  M(s)W(s) dW(s)  \right]=&
	\mathbb{E}\left[ \int_0^t M^2(s)B(s)W(s) \rho ds \right]\\
	=& \rho \int_0^tM^2(s)  \mathbb{E}[B(s)W(s)]ds,\\
	=& \rho^2 \int_0^t s M^2(s) ds.
	\end{align*}
	Since $W^2(t) = t + 2\int_0^tW(s)dW(s) $ we obtain 
	\begin{align*}
	\mathbb{E}\left[W^2(t)\int_0^t  M(s)B(s) dB(s)\right]=& t\mathbb{E}\left[\int_0^t  M(s)B(s) dB(s) \right]\\&+ 2\mathbb{E}\left[\int_0^t W(s)dW(s) \int_0^t M(s)B(s)dB(s)  \right]\\
	=& 2 \mathbb{E}\left[ \int_0^t M(s)W(s)B(s) \rho ds \right] \\
	=&2\rho^2 \int_0^t sM(s) ds .
	\end{align*}
	Similarly, 
	\begin{align*}
	\mathbb{E}\left[B^2(t)  \int_0^t  M(s)W(s) dW(s)\right]=
	2\rho^2\int_0^t sM(s) ds .
	\end{align*}
	Hence, the identity (\ref{eqn: app.first}) holds. 
\end{proof}
\begin{lemma} \label{lemma: app.main2}
	Let $W_1$ and $W_2$ be two independent one-dimensional standard Brownian motions. Then $\int_0^t \mu(s)W_1(s)W_2(s)ds$ has the following second moment
	\begin{align*}
	\mathbb{E}\left[\left( 	\int_0^t W_1(s)W_2(s)  \mu(s) ds\right)^2\right] =& \frac{1}{2}\left(\int_0^t s \mu(s) ds\right)^2 + \frac{1}{2}t^2M^2(t) + 2 \int_0^t s M^2 (s) ds \\ 
	&- \frac{1}{2}\left(\int_0^t M(s) ds\right)^2 - 4M(t)\int_0^t sM(s)ds\\& + t M(t)\int_0^t M(s)ds.
	\end{align*}
\end{lemma}
\begin{proof}
	Let	 $B(t) = \frac{W_1(t)+W_2(t)}{\sqrt{2}}$, clearly it is a standard Brownian motion such that  $\rho_{B(t), W_1(t)} = \rho_{B(t), W_2(t)} = \frac{1}{\sqrt{2}}.$
	Lemma \ref{lemma: app.A.3} yields that
	\begin{align*}
	\int_0^t W_1(s)W_2(s) \mu(s) ds = \int_0^t  B^2(s)  \mu(s)ds - \frac{1}{2}\int_0^t W_1^2(s)  \mu(s)ds -\frac{1}{2}\int_0^t  W_2^2(s)  \mu(s)ds. 
	\end{align*}
	Thus, we obtain 
	\begin{align*}
	\mathbb{E}\left[ \left( 	\int_0^t  W_1(s)W_2(s)  \mu(s)ds \right)^2\right] = & \mathbb{E}\left[\left(\int_0^t B^2(s)  \mu(s)ds\right)^2 \right] +\frac{1}{4}\mathbb{E}\left[\left(\int_0^t  W_1^2(s)  \mu(s)ds\right)^2 \right]\\&+\frac{1}{4}\mathbb{E}\left[\left(\int_0^t  W_2^2(s)  \mu(s)ds\right)^2 \right]\\&-\mathbb{E}\left[\int_0^t  B^2(s)  \mu(s)ds\int_0^t  W_2^2(s)  \mu(s)ds \right]\\&+\frac{1}{2}\mathbb{E}\left[\int_0^t W_1^2(s)  \mu(s)ds \right] \mathbb{E}\left[\int_0^t  W_2^2(s)  \mu(s)ds\right]\\
	&- \mathbb{E}\left[\int_0^t  B^2(s) \mu(s) ds\int_0^t W_1^2(s) \mu(s)  ds \right]
	\\=&
	\frac{3}{2} \mathbb{E}\left[\left(\int_0^t B^2(s)  \mu(s) ds\right)^2 \right] + \frac{1}{2}\left(\int_0^t s \mu(s) ds \right)^2 \\&-2\mathbb{E}\left[\int_0^t  B^2(s) \mu(s)  ds\int_0^t W_1^2(s)  \mu(s) ds \right].
	\end{align*}
	Employing Lemma \ref{lemma: app.main1} and Lemma \ref{lemma: app.A.4} we can conclude that $\mathbb{E}\left[ \left( 	\int_0^t  W_1(s)W_2(s)  \mu(s) ds \right)^2\right]$ has the desired form. 
\end{proof}
\end{appendices}
\medskip
\section*{References}
\addcontentsline{toc}{section}{References}%
\bibliographystyle{unsrt}
\bibliography{MRWRefs}

\end{document}